\def\misajour{August 23,  2011}
 
\documentclass{jaums}

\usepackage[applemac]{inputenc}

\theoremstyle{thmit} 
\newtheorem{thm}{Theorem}[section]
\newtheorem{lem}[thm]{Lemma}
\newtheorem{cor}[thm]{Corollary}
\newtheorem{prop}[thm]{Proposition}

\theoremstyle{thmrm} 
\newtheorem{exa}{Example}

\newtheorem*{oldproof}{Proof}
\renewenvironment{proof}[1][{}]{\begin{oldproof}[#1]}{\qed\end{oldproof}}

\usepackage[pdftex]{graphicx} 
\usepackage{hyperref}

\def\C{\mathbf{C}} 
\def\N{\mathbf{N}} 
 
\def\Q{\mathbf{Q}} 
\def\R{\mathbf{R}}

\def\Z{\mathbf{Z}}
\def\Qbar{\overline{\Q}}
\def\Tr{\mathrm{Tr}}
\newcommand{\ZK}{\Z_K}

\def\rmh{{\mathrm{h}}}
 
\def\calB{\mathcal {B}}

 \def\Id{{\mathrm{Id}}}

\def\house#1{\setbox1=\hbox{$\,#1\,$}%
\dimen1=\ht1 \advance\dimen1 by 2pt \dimen2=\dp1 \advance\dimen2 by 2pt
\setbox1=\hbox{\vrule height\dimen1 depth\dimen2\box1\vrule}%
\setbox1=\vbox{\hrule\box1}%
\advance\dimen1 by .4pt \ht1=\dimen1
\advance\dimen2 by .4pt \dp1=\dimen2 \box1\relax}

\title{Approximation of an algebraic number by products of rational numbers and units }
\author{Claude LEVESQUE}
\address{D\'{e}partement de math\'{e}matiques et de statistique,\\
 Universit\'{e} Laval,\\
Qu\'{e}bec (Qu\'{e}bec),\\
CANADA G1V 0A6\\
\href{Claude.Levesque@mat.ulaval.ca}{\tt Claude.Levesque@mat.ulaval.ca}}
\author{Michel WALDSCHMIDT}
\address{Institut de Math\'{e}matiques de Jussieu, \\
Universit\'{e} Pierre et Marie Curie (Paris 6),\\
4 Place Jussieu, \\
F -- 75252 PARIS Cedex 05, FRANCE\\
\href{miw@math.jussieu.fr}{\tt miw@math.jussieu.fr}\\
\href{http://www.math.jussieu.fr/~miw/}{\tt http://www.math.jussieu.fr/$\sim$miw/}
 }

\begin{document}
 
\maketitle

 \null
 
 \vskip - 7 true cm 
\hfill
{\small date de mise à jour: \it \misajour}

\null
 \vskip 5 true cm 
\hfill{\it Dedicated to Alf van der Poorten}
\bigskip

\begin{abstract} 
We relate a previous result of ours on families of diophantine equations having only trivial solutions 
with a result on the approximation of an algebraic number by products of rational numbers and units. We compare this approximation, on the one hand with a Liouville type estimate, on the other hand with an estimate arising from a lower bound for a linear combination of logarithms.

\end{abstract}
 \bigskip
 
{\sc American Mathematical Society 2010 Mathematics subject classification}: 
11D59, 
11J87, 
11J68, 
11J86, 
11J17 
\\
\null\qquad 

\section{Introduction}\label{S:Introduction}
 
 In a previous paper \cite{LW}, we proved that certain families of diophantine equations have only trivial solutions. 
 In this note (Theorem \ref{thm:principal}), we show how to deduce from our results on families of diophantine equations \cite{LW} some results on the approximation of an algebraic number by products of rational numbers and units. Since the proofs rest on Schmidt's Subspace Theorem, these results are non--effective. They improve elementary estimates (Proposition \ref{Proposition:VarianteLiouville}) obtained along the lines of Liouville's arguments. A different type of estimate (Theorem $\ref{thm:Effectif}$), which is effective, is achieved by means of a lower bound for a linear combination of logarithms of algebraic numbers.

\section{ A variant of the Liouville inequality }\label{S:VarianteLiouville}

Rational numbers will be written $p/q$ with $q>0$ and $\gcd(p,q)=1$, 
(hence with $q=1$ in case $p=0$). 
When $\alpha$ is an algebraic number of degree $d$, its minimal polynomial will be denoted by 
\begin{equation}\label{Equation:PolynomeMinimal}
f(X)=a_0X^d+a_1X^{d-1}+\cdots+a_d\in\Z[X]
\end{equation}
where $a_0>0$. In $\C[X]$, this polynomial splits as 
 $$
f(X)\;=\;a_0\prod_{\sigma}(X-\sigma(\alpha)) 
$$
where $\sigma$ in the product runs through the set of embeddings of the field $ K:=\Q(\alpha)$ in $\C$. 
We denote by $\house{\alpha}$ the maximum complex modulus of the algebraic conjuguates of $\alpha$ in $\C$:
$$
\house{\alpha}=\max_{\sigma} |\sigma(\alpha)|. 
$$ 
The absolute logarithmic height of $\alpha$ (see \cite{GL326}, Chap.~3) is 
 $$
 \rmh(\alpha)=\log a_0+\sum_{\sigma} \log \max \bigl\{1,\; |\sigma(\alpha)| \bigr\}.
 $$
The proof of the next result is close to that of Liouville. 
 
\begin{prop} 
 \label{Proposition:VarianteLiouville}
Let $\alpha\in\C$ be an algebraic number of degree $d$ with minimal polynomial $(\ref{Equation:PolynomeMinimal}$). Then for any $p/q\in\Q$ and for any unit $\varepsilon$ of $\Q(\alpha)$ such that $\varepsilon\alpha\not= p/q$, we have
$$
\left | \alpha - \frac{p}{q} \cdot \varepsilon^{-1}\right | 
\geq \frac{\kappa_1}{q^d \; \,{\house{\varepsilon}\,}^{d-1}|\varepsilon|}, 
\quad \mbox{with}\quad 
\kappa_1=\frac{1}{a_0(2\, \house{\alpha}\, +1 )^{d-1}}\cdotp
$$
\end{prop} 

The conclusion can also be written: {\it for any $p/q\in\Q$ and for any unit $\varepsilon$ of $\Q(\alpha)$ such that $\varepsilon\alpha\not= p/q$, we have}
$$ 
 | q \varepsilon\alpha - p | 
\geq \frac{\kappa_1}{ (q \, \house{\varepsilon}\,)^{d-1}}\cdotp
$$ 
Proposition $\ref{Proposition:VarianteLiouville}$  holds for any algebraic integer $\varepsilon$ lying in $\Q(\alpha)$, not only units, provided that we take into account the leading coefficient  of the minimal polynomial of $\varepsilon$.  Indeed, from Proposition 3.14 of  \cite{GL326} (which follows from the fact hat the norm of a non--zero rational integer has absolute value at least $1$ and which also follows from the product formula), one deduces that for any algebraic number field $K$ of degree $d$, any element $\gamma$ in $K$ and any rational number $p/q$ with $q\gamma\not=p$, we have 
$$ 
 | q \gamma- p | 
\geq 
\frac{1}{ (|p|+q)^{d-1} e^{d\rmh(\gamma)}}\cdotp
$$ 
We prefer to restrict the situation in Proposition $\ref{Proposition:VarianteLiouville}$   to the special case where $\gamma= \varepsilon\alpha$ for the sake of comparison with Theorem $\ref{thm:principal}$  andTheorem $\ref{thm:Effectif}$.

\begin{proof} 
We assume that the rational number $p/q\in\Q$ and the unit $\varepsilon$ of $\Q(\alpha)$ satisfy $\varepsilon\alpha\not= p/q$ and we aim to prove
 (a third formulation for) the conclusion of Proposition $\ref{Proposition:VarianteLiouville}$, namely
$$ 
\left | \varepsilon\alpha - \frac{p}{q} \right | 
\geq \frac{\kappa_1}{ q^d {\, \house{\varepsilon}\,}^{d-1}}\cdotp
$$
As we have just seen, this will complete the proof of Proposition $\ref{Proposition:VarianteLiouville}$.

We consider the embeddings $\sigma$ of $\Q(\alpha)$ into $\C$; since $\alpha\in\C$, one of them is the inclusion of $\Q(\alpha)$ in $\C$, that we denote $\Id$. Letting 
$$
f_\varepsilon(X)=a_0\prod_{\sigma}(X-\sigma(\varepsilon\alpha))
\quad
\hbox{and}
\quad
F_\varepsilon(X,Y)=Y^df_\varepsilon(X/Y),
$$
we have 
\begin{equation}\label{Equation:Fpq}
 |F_\varepsilon(p,q)|\;=\; a_0q^d \left|\varepsilon\alpha-\frac{p}{q}\right|\prod_{\sigma\neq \Id} \left|\sigma(\varepsilon\alpha)-\frac{p}{q}\right|.
\end{equation}
Since $q\ge 1$, $\kappa_1<1$ and $\house{\varepsilon}\, \ge 1$, the inequality that we want to establish is trivial if $|\varepsilon \alpha-p/q|\ge 1$. Therefore we can assume $|\varepsilon \alpha-p/q|< 1$,
in which case we have, for every $\sigma$, 
$$
 \left|\sigma(\varepsilon\alpha)-\frac{p}{q}\right|\;\leq\; |\sigma(\varepsilon\alpha) - \varepsilon\alpha |+1\le 2\, \house{\varepsilon\alpha}\, +1 \;\leq\;
(2\, \house{\alpha}\, +1 )\, \house{\varepsilon}.
$$ 
By assumption $\varepsilon\alpha\not= p/q$; hence, for each $\sigma$, the Galois conjugate $\sigma(\varepsilon\alpha)$ of $\varepsilon\alpha$ is distinct from $\sigma(p/q)= p/q$, and therefore
 $F_\varepsilon(p,q)\not=0$. Since $F_\varepsilon(p,q)\in\Z$, we have $ |F_\varepsilon(p,q)|\ge 1$, and the conclusion follows. 
\\
 \end{proof}
 
Proposition $\ref{Proposition:VarianteLiouville}$ is trivial when $d=1$. 
In Section $\ref{S:ResultatPrincipal}$, we will show that this result is not optimal when $d\geq 3$, in the sense that we can replace $\kappa_1$ by an arbitrarily large constant, provided that $q$ be sufficiently large. Consider the case $d=2$. Let $\alpha$ 
be a quadratic number. The conclusion of Proposition $\ref{Proposition:VarianteLiouville}$ is still trivial if $\alpha$ is not real. So we suppose $\alpha \in \R$ and we denote by $\epsilon_0$ the fundamental unit $>1$ of the real quadratic field $\Q(\alpha)$. We plan to investigate how close to a best possible one is the lower bound exhibited in Proposition \ref{Proposition:VarianteLiouville}. For ease of notation we assume $\varepsilon>0$ without loss of generality. There are two cases: if $\varepsilon >1$, then we write $\varepsilon=\epsilon_0^n$ with $n>0$, while if 
$\varepsilon <1$, then we write $\varepsilon=\epsilon_0^{-n}$ with again $n>0$. From Proposition \ref{Proposition:VarianteLiouville} we infer, for all $p/q\in\Q$ and all $n>0$, 
$$
\left | \epsilon_0^n\alpha - \frac{p}{q}\right | 
\ge \frac{\kappa_1}{q^2 \epsilon_0^n } 
\quad\hbox{and}\quad 
\left | \epsilon_0^{-n}\alpha - \frac{p}{q}\right | 
\ge \frac{\kappa_1}{q^2 \epsilon_0^n }\cdotp
$$
 The next result shows that, infinitely often, these estimates cannot be improved: the dependence on $q$ is sharp in the quadratic case.

 \begin{lem}\label{lem:quadratic}
 Let $\epsilon_0$ be the fundamental unit $>1$ of the real quadratic field $\Q (\alpha )$.
 For any $n\ge 0$ with at most one exception, there exists a constant $\kappa_2$ and infinitely many rational numbers $p/q$ 
such that 
$$
\left | \epsilon_0^n\alpha - \frac{p}{q}\right | 
\le \frac{\kappa_2}{q^2 \epsilon_0^n } 
$$
 and infinitely many rational numbers $p/q$ 
such that 
$$
\left | \epsilon_0^{-n} \alpha - \frac{p}{q}\right | 
\le \frac{\kappa_2}{q^2 \epsilon_0^n } \cdotp
$$
An admissible value for the constant 
is $\kappa_2=\epsilon_0^n/\sqrt{5}$.
\end{lem}

However this lemma does not give a satisfactory answer to the question of optimality, because $n$ is fixed and 
$\kappa_2$ 
depends on $n$. In \S$\ref{S:Effectif}$ we show that the dependence on $\house{\varepsilon}$ is not sharp in Liouville's estimate, even in the quadratic case.

\begin{proof}[of Lemma \ref{lem:quadratic}]
The possible exception is $n$ such that $\epsilon_0^n \alpha$ or $\epsilon_0^{-n} \alpha$ is rational, and the result follows 
from a theorem of Hurwitz (see for instance \cite{LN785}, Chap.~1, Th. 2F).
\end{proof}

\section{A refinement of Liouville's estimate}\label{S:ResultatPrincipal}

Let $\alpha$ be an algebraic number of degree $d$ over $\Q$ and let $K$ be the field $\Q(\alpha)$.
In this section, we plan to prove the following result.

\begin{thm} 
 \label{thm:principal}
 For any constant $\kappa>0$, the set of pairs $(p/q,\varepsilon)\in\Q\times\ZK^\times$ such that 
$[\Q(\varepsilon\alpha):\Q]\ge 3$ and 
\begin{equation}\label{Equation:TheoremePrincipal}
\left | \varepsilon\alpha - \frac{p}{q}\right | \le \frac{\kappa}{q^d \; \house{\varepsilon}^{\, d-1}}
\end{equation}
is finite.
\end{thm}

 Theorem $\ref{thm:principal}$ is trivial when $d=1$ and when $d=2$. For the proof we can suppose $d\ge 3$. 
For $\kappa<\kappa_1$, the inequality   $(\ref{Equation:TheoremePrincipal})$ is a consequence of   Proposition $\ref{Proposition:VarianteLiouville}$, and the set of exceptional $(p/q,\varepsilon)$ has at most one element. The point    is that  Theorem   $\ref{thm:principal}$  holds true for any arbitrarily large constant $\kappa$. The conclusion can also be stated  the following way
$$
  (q\; \house{\varepsilon})^{d-1} \Vert q \varepsilon\Vert  \longrightarrow +\infty
\quad\hbox{as}\quad \max\{q,\house{\varepsilon}\}   \longrightarrow +\infty.
$$

We twist the minimal polynomial $(\ref{Equation:PolynomeMinimal}$) of $\alpha$ by a unit $\varepsilon$ of $K$ by writing
$$
f_\varepsilon(X)=a_0\prod_{\sigma}(X-\sigma(\varepsilon\alpha))
\quad
\hbox{and}
\quad
F_\varepsilon(X,Y)=Y^df_\varepsilon(X/Y).
$$
 Theorem $\ref{thm:principal}$ is a corollary of the following theorem whose proof can be found in \cite{LW}.

\begin{thm} 
 \label{thm:LW}
 For any integer $k\not=0$, the set of triples $(x,y,\varepsilon)\in\Z^2\times\ZK^\times$ with $xy\not=0$ satisfying 
$$
[\Q(\varepsilon\alpha):\Q]\ge 3 
 \quad\hbox{and}\quad F_\varepsilon (x,y)=k
$$
is finite.
\end{thm}


\begin{proof} [of Theorem $\ref{thm:principal}$]
Let $\kappa>0$ and let $(p/q,\varepsilon)\in\Q\times \ZK^\times$ verify $(\ref{Equation:TheoremePrincipal})$. We have 
 $ \house{\varepsilon}\geq 1$. There is no restriction in supposing $q^d\ge \kappa$. 
Consider the relation $(\ref{Equation:Fpq})$. 
For $\sigma\neq \Id$, we use the upper bound 
$$
 \left|\sigma(\varepsilon\alpha)-\frac{p}{q}\right|\;\leq\; |\varepsilon\alpha-\sigma(\varepsilon\alpha)|+ \left| \varepsilon\alpha-\frac{p}{q}\right| 
\; \leq \; |\varepsilon\alpha-\sigma(\varepsilon\alpha)| +1,
 $$
which is coming from a weak form of $(\ref{Equation:TheoremePrincipal})$. Since
$$
 |\varepsilon\alpha-\sigma(\varepsilon\alpha)| +1 \; \leq\; 2\; \house{\varepsilon\alpha}+1\;\leq\; (2\; \house{\alpha}+1 )\; 
\house{\varepsilon}\, ,
$$
we deduce, by taking into account $(\ref{Equation:TheoremePrincipal})$,
$$
|F_\varepsilon(p,q)|\;\leq\; a_0 q^d \left| \varepsilon\alpha-\frac{p}{q}\right| (2\; \house{\alpha}+1 )^{d-1}\; { \house{\varepsilon}\; }^{d-1}\; \leq\; 
a_0 \kappa (2\; \house{\alpha}+1 )^{d-1}.
$$
Theorem $\ref{thm:LW}$ allow us to conclude that the set of rational numbers $p/q$ which satisfy  $(\ref{Equation:TheoremePrincipal})$  is finite. 
 \end{proof}

To illustrate Theorem $\ref{thm:principal}$, let us make explicit the case of a cubic field whose unit group is of rank $1$.

\begin{cor}
 \label{corollaire:cubique}
Let $\alpha$ be a real cubic number which has two Galois imaginary conjuguates. Let $\epsilon_0$ be the real fundamental unit 
 $> 1$ of the cubic field $\Q(\alpha)$. 
For any constant $\kappa>0$, the set of pairs $(p/q,n)\in\Q\times\Z$ with $n>0$, such that 
$$
\left| \alpha-\frac{p}{q\epsilon_0^n}\right| \le \frac{\kappa}{q^3\epsilon_0^{3n}}
$$
is finite, and the set of pairs $(p/q,n)\in\Q\times\Z$ with $n>0$, such that 
$$
\left| \alpha-\frac{p\epsilon_0^n}{q}\right| \le \frac{ \kappa }{q^3}
$$
is finite.
\end{cor}

\begin{proof}
In Theorem $\ref{thm:principal}$, take $\varepsilon=\epsilon_0^n$ with $\house{\varepsilon}=\epsilon_0^n$ for the first part of the statement of Corollary $\ref{corollaire:cubique}$, $\varepsilon=\epsilon_0^{-n}$ with $\house{\varepsilon}=\epsilon_0^{n/2}$ for the second one. 
\end{proof}

\section{Effective estimates}\label{S:Effectif}

 A sharp estimate from below for $|\varepsilon\alpha-p/q|$ in terms of $\house{\varepsilon}$ can be achieved in an effective way by means of a lower bound for linear combinations of logarithms of algebraic numbers. 
 
 \begin{thm}\label{thm:Effectif}
 Let $K$ be a number field and let $\alpha\in K$. There exists an effectively computable constant $\kappa_3>0$ such that, for any unit $\varepsilon\in\ZK^\times$ and any rational number $p/q$ with $\varepsilon\alpha\not=p/q$, 
\begin{equation}\label{Equation:ThmEffectif}
\left| \varepsilon\alpha-\frac{p}{q}\right| \ge 
\bigl(\log(\house{\varepsilon}+2)
\bigr)^{-\kappa_3 \log \max\{|p|,\; q,\; 2\}}.
\end{equation}
 \end{thm}

 We  will easily deduce Theorem $\ref{thm:Effectif}$  as a  consequence of Proposition 9.21 of  \cite{GL326}, but we can also deduce it 
from Theorem 4.1 of  \cite{BEG} with an explicit value for $\kappa_4$. At the same time we notice that  Theorem $\ref{thm:Effectif}$ can be generalized to groups of $S$--units of a number field in place of $\ZK^\times$, which amounts to replace $\ZK^\times$ by any finitely generated subgroup of the multiplicative group of a fixed number field.


 \begin{prop}\label{prop:MinorationCLL}
Let $m$ and $D$ be positive integers. There exists an effectively computable 
positive number $\kappa_4$ depending only on $m$ and $D$ with the following property.
Let $\lambda_1,\ldots,\lambda_m$ be
logarithms of algebraic numbers; define $\alpha_j=\exp(\lambda_j)$ \
($1\le j\le m$). Assume that the number field
$\Q(\alpha_1,\ldots,\alpha_m)$ has degree $\le D$ over
$\Q$. Let $b_1,\ldots,b_m$ be rational integers, not all of which are zero. Further, let $B, A_1,\ldots,A_m$ be positive
real numbers. Assume
$$
\log A_j\ge\max\left\{\rmh(\alpha_j),\ |\lambda_j|,\ 1\right\}
\quad (1\le j\le m),
$$ 
and 
$$
B\ge\max \{|b_1|,\ldots, |b_m|, e\} .
$$ 
Assume further that the number
$$
\Lambda:= b_1\lambda_1+\cdots+b_m\lambda_m
$$
is nonzero. Then
$$
|\Lambda|>
\exp\{- 
\kappa_4(\log B)(\log A_1)\cdots(\log A_m) \}. 
$$
 \end{prop}
 
The following auxiliary lemma will also be used in the proof of Theorem $\ref{thm:Effectif}$. 

\begin{lem}\label{lemma:EquivalenceNormes}
Let $K$ be a number field of degree $d=r_1+2r_2$ and unit group of rank $r=r_1+r_2-1$, where $r_1$ is the number of real embeddings of $K$ into $\R$ and $r_2$ 
is the number of pairs of non–real  embeddings   of   
$K$ into $\C$. Let $\epsilon_1,\ldots,\epsilon_r$ be a basis of the torsion--free part of the group of units $\Z_K^\times$ of $K$. Then there is a constant $\kappa_8>0$ such that, 
for any unit $\varepsilon$ of K written as
$$ 
 \varepsilon=\zeta\epsilon_1^{b_1}\cdots\epsilon_r^{b_r}
 $$
 where $\zeta$ is a root of unity in $K$ and $b_1,\ldots,b_r$ are rational integers, we have
\begin{equation}\label{Equation:EquivalenceNormes}
 \max\{|b_1|,\ldots,|b_r|\}\le \kappa_8 \log \house{\varepsilon}.
\end{equation}
\end{lem}

 \begin{proof}[of Lemma $\ref{lemma:EquivalenceNormes}$]
 Consider the logarithmic embedding $\lambda$ of $K^\times$ in $\R^{r_1+r_2}$ given by 
 $$
 \lambda (\alpha)=\bigl(\log|\sigma_i(\alpha)|\bigr)_{1\le i\le r_1+r_2}, 
 $$
 where $\sigma_1,\ldots,\sigma_{r_1}$
 are the real embeddings of $K$ into $\R$ and $\sigma_{r_1+1},\ldots,\sigma_{r_1+r_2}$ are the pairwise non--conjugate non--real embeddings of $K$ into $\C$.
 Denote by $\Vert \cdot \Vert_1$ the sup norm on $\R^{r_1+r_2}$, so that, for $\alpha\in K^\times$, we have
 $$
 \Vert \lambda (\alpha)\Vert_1= \log \house {\alpha}.
 $$
 The image of $\Z_K^\times$ under $\lambda$ is a lattice in the hyperplane $H$ of equation 
 $$
 x_1+\cdots+x_{r_1}+ 2x_{r_1+1}+\cdots+2x_{r_1+r_2}=0.
 $$
 A basis of $H$ is $\calB=\{\lambda(\epsilon_1),\ldots,\lambda(\epsilon_r)\}$. 
We denote by $\Vert \cdot \Vert_2$ the sup norm on $H$ with respect to the basis $\calB$. For $\varepsilon\in\Z_K^\times$ written as 
$$
 \varepsilon=\zeta\epsilon_1^{b_1}\cdots\epsilon_r^{b_r},
 $$
 where $\zeta$ is a root of unity in $K$ and $b_1,\ldots,b_r$ are rational integers, we have
 $$
 \Vert \lambda(\varepsilon) \Vert_2= \max\{|b_1|,\ldots,|b_r|\}.
 $$ 
Lemma $\ref{lemma:EquivalenceNormes}$ follows from the equivalence of the norm $\Vert\cdot\Vert_2$ and the restriction to $H$ of the norm $\Vert\cdot\Vert_1$. More explicitly, we deduce 
$(\ref{Equation:EquivalenceNormes})$ by looking at the absolute values 
of the $b_i$'s obtained via Cramer's formulas for the solutions of the system of linear equations
$$
b_1\log |\sigma_j(\epsilon_1)|+\cdots +b_r\log |\sigma_j(\epsilon_r)|= \log|\sigma_j(\varepsilon)|\quad(j=1,\ldots,
r_1+r_2),
$$
which has rank $r$ since the regulator of $K$ does not vanish. 
 \end{proof}

 \begin{proof}[of Theorem $\ref{thm:Effectif}$]
 The estimate 
 $(\ref{Equation:ThmEffectif})$ we want to prove is trivial in the case 
 $$
 \left| \varepsilon\alpha-\frac{p}{q}\right| \ge \frac{|p|}{2q},
 $$
 hence we may assume that the number $\gamma:=\varepsilon\alpha q/p$ satisfies 
 $$
 0< |\gamma -1| < \frac{1}{2},
 $$
 and therefore the principal value $\lambda_0$ of the logarithm of $\gamma$ satisfies (see \cite{GL326}, Exercise 1.1.b):
\begin{equation}\label{equation:lambda}
0<|\lambda_0| < 2| \gamma -1|.
 \end{equation}
Let $\epsilon_1,\ldots,\epsilon_r$ be a basis of the torsion--free part of the group $\Z_K^\times$ of units of $K$. Write 
 $$
 \varepsilon=\zeta\epsilon_1^{b_1}\cdots\epsilon_r^{b_r}
 $$
 where $\zeta$ is a root of unity in $K$ and $b_1,\ldots,b_r$ are rational integers. For $1\le j\le r$, select a logarithm $\log\epsilon_j$ of $\epsilon_j$, and set 
 $$
 \lambda_{r+1}= \lambda_0-b_1\log\epsilon_1-\cdots - b_r\log\epsilon_r, 
$$
so that $e^{\lambda_{r+1}}=\zeta\alpha q/p$. 
 We use Proposition $\ref{prop:MinorationCLL}$ with $m=r+1$, $\lambda_j=\log\epsilon_j$ for $1\le j\le r$ and $b_{r+1}= 1$. 
 The number $\kappa_4$  is a constant depending only on $\alpha$ and $K$, and we may choose for $A_1,\ldots,A_r$ constants which also depend only on $\alpha$ and $K$. 
 Moreover, for $A_{r+1}$ and $B$, we take 
 $$
 A_{r+1}=\kappa_5\max\{|p|,\; q,\; 2\}\quad\hbox{and}
 \quad
 B=\kappa_6 \log(\house{\varepsilon}+1),
 $$
 where again $\kappa_5$ and $\kappa_6$ are constants depending only on $\alpha$ and $K$. The upper bound for 
 $\max\{|b_1|,\dots,|b_r|\}$ 
 follows from Lemma \ref{lemma:EquivalenceNormes}. 
 We deduce that there exists a constant $\kappa_7$, depending only on $\alpha$ and $K$, such that 
\begin{equation}\label{equation:formelineaire}
 |\lambda_0| =
 \bigl | b_1\log\epsilon_1+\cdots + b_r\log\epsilon_r+\lambda_{r+1} \bigr | \ge 
\exp\bigl\{-\kappa_7(\log B)\log \max\{|p|,\; q,\; 2\}\bigr\}.
\end{equation}
The result easily follows from $(\ref{equation:lambda})$ and $(\ref{equation:formelineaire})$.
\end{proof}

 \section{Comparison with a result of Corvaja and Zannier}

 \medskip
 Denote by $\Vert\cdot\Vert$ the distance to the nearest integer: for $x\in\R$, 
 $$
 \Vert x \Vert :=\min_{n\in\Z} |x-n|.
 $$ 
 Let $\Qbar $ denote the field of complex  numbers which are algebraic  over $\Q$.  
Following \cite{CZ}, call a (complex) algebraic number $\xi$ a {\it pseudo--Pisot} number if 
\\
$(i)$
$|\xi|>1$ and all its conjugates have (complex) absolute value strictly less than $1$;
\\
$(ii)$
$\xi$ has integral trace: $\Tr_{\Q(\xi)/\Q}(\xi)\in \Z$.

The main Theorem of Corvaja and Zannier in \cite{CZ}, whose proof also rests on Schmidt's Subspace Theorem,  can be stated as follows. 

\begin{thm}\label{thm:CZ}
Let $\Gamma\subset\Qbar^\times$ be a finitely generated multiplicative group of algebraic numbers, let $\alpha\in\Qbar^\times$ be a non--zero algebraic number and let $\eta>0$ be fixed. Then there are only finitely many pairs $(q,\varepsilon)\in \Z\times\Gamma$ with $\delta=[\Q(\varepsilon):\Q]$ such that $|\alpha q\varepsilon|>1$, $\alpha q\varepsilon$ is not a pseudo--Pisot number and 
\begin{equation}\label{Equation:PC}
0<
\Vert \alpha q\varepsilon\Vert 
< \frac{1}{e^{\eta \rmh(\varepsilon)} q^{\delta+\eta}}\cdotp 
\end{equation}
\end{thm}

The special case $\varepsilon=1$, $\delta=1$ of Theorem $\ref{thm:CZ}$ is a Roth--type estimate. The proof we gave in \S~$\ref{S:ResultatPrincipal}$ relies on our result on Diophantine equations in \cite{LW}, which is a consequence of Schmidt's Subspace Theorem, while the proof of Corvaja and Zannier in \cite{CZ} uses directly Schmidt's fundamental result on linear forms in algebraic numbers. It is likely that an improvement of our result could be achieved by adapting the arguments of  \cite{CZ} -- so one would expect to obtain a refinement of our conclusion which would also include the statement of Theorem $\ref{thm:CZ}$.

However it turns out that in some very particular cases,  Theorem $\ref{thm:CZ}$ is weaker than Liouville's estimate  (Proposition $\ref{Proposition:VarianteLiouville}$), hence weaker than our Theorem $\ref{thm:principal}$. 
Here is an example.  Assume in Theorem $\ref{thm:CZ}$ that  $\Gamma$ is the group of units $\ZK^\times$ of a number field $K$ of degree $d=\delta$ and that  $\alpha\in K$. In this special case, for  $\varepsilon\in\Z_K^\times$, we may replace $\log\house{\varepsilon} $ by $ \rmh(\varepsilon)$ without spoiling the result, since 
$$
\log\house{\varepsilon}\le \rmh(\varepsilon)\le d \log\house{\varepsilon}.
$$  
Hence Theorem $\ref{thm:CZ}$ implies that for any $\eta>0$,  there are only finitely many pairs $(q,\varepsilon)\in \Z\times\Z_K^\times$  such that $|\alpha q\varepsilon|>1$, $\alpha q\varepsilon$ being not a pseudo--Pisot number and 
$$
0<
\Vert \alpha q\varepsilon\Vert 
< \frac{1}{ {\house{\varepsilon} \,}^{\eta} q^{d+\eta}}\cdotp
$$
In other words, if  $|\alpha q \varepsilon | > 1$ with  $\alpha q \varepsilon$  being not a pseudo–Pisot number,  then for all pairs   $(q, \varepsilon) \in \Z \times  \Z_K^\times$ ,   except for finitely many of them, we have 
\begin{equation}\label{equation:CZ}
\Vert \alpha q\varepsilon\Vert  \geq  \frac{1}{ {\house{\varepsilon} \,}^{\eta} q^{d+\eta}}\cdotp
\end{equation}
It may be observed that a more concise form of this statement is 
$$
\liminf   \frac{d\log q + \log \Vert \alpha q\varepsilon\Vert }{\log q +
\log \house{\varepsilon}
}\ge  0 \quad\hbox{as}\quad \max\{q,\house{\varepsilon}\}   \longrightarrow +\infty,
$$%
where  $(q,\varepsilon)\in \Z\times \Z_K^\times $ with   $|\alpha q\varepsilon|>1$ and  $\alpha q\varepsilon$ being not a pseudo--Pisot number.


 In the case where the pairs $(q,\varepsilon)$ 
belong to a set  in which ${\house{\varepsilon}\;}^{d-1}q^{-1}$ is bounded from above, $(\ref{equation:CZ})$ is weaker than the lower bound 
$$
(q \, \house\varepsilon)^{d-1}\Vert q\alpha\varepsilon \Vert \ge \kappa_1
$$
given by Liouville's inequality (Proposition $\ref{Proposition:VarianteLiouville}$), hence it is weaker than the result which one deduces from Theorem $\ref{thm:principal}$. 

For the comparison with  $(\ref{Equation:ThmEffectif})$, let us consider a set of pairs $(q,\varepsilon)$ 
in which $(\log|\varepsilon|)/\log q$ is bounded from above by a positive constant and at the same time $\log \house{\varepsilon} (\log q)^{-1}(\log \log q)^{-1}$ is bounded from below by a positive constant.  In this case one deduces from  Theorems $\ref{thm:Effectif}$
$$
\Vert \alpha q\varepsilon
\Vert
\ge 
\exp\bigl\{
-\kappa_5
\bigl(\log \log(\house{\varepsilon}+2)
\bigr) \log \max\{|q\varepsilon|,\; 2\}
\bigr\}
$$
where $\kappa_5$ is an effectively computable constant depending only on $\alpha$ and $K$. Hence in this special case,  the lower bound for $\Vert \alpha q\varepsilon\Vert$ which we deduce from  $(\ref{Equation:ThmEffectif})$ is a power of $ \log(\house{\varepsilon})$, while 
 Theorem  $\ref{thm:CZ}$   yields a weaker lower bound, namely  a power of $ \house{\varepsilon}$. 
Thus 
 Theorem $\ref{thm:principal}$   sometimes yields  sharper estimates than  Theorems $\ref{thm:Effectif}$ and $\ref{thm:CZ}$ 
 when $q$ is large, 
 Theorem $\ref{thm:Effectif}$ is  effective and may be sharper than  Theorems  $\ref{thm:principal}$ and $\ref{thm:CZ}$  when $\house{\varepsilon}$ is large, while 
 Theorem $\ref{thm:CZ}$ is most often sharper than $\ref{thm:principal}$ and  $\ref{thm:Effectif}$
  for an intermediate range. However, we emphasize the fact that  Theorem $\ref{thm:CZ}$ has a wider scope, even if it happens to be sometimes less strong than other results.

\medskip
We conclude with two selected examples in which we take  $\alpha =1$. 
 
\begin{exa} Consider a cubic number field $K$ with group of units of rank $1$ and let $\epsilon_0 >1$ be  the fundamental unit, so that $\Z_K^\times = \{1, -1\} \times \langle \epsilon_0 \rangle$ and $\house{\epsilon_0}=\epsilon_0$. 
Theorem  $\ref{thm:principal}$  states that for any $\kappa>0$, there are only finitely many $(n,q) \in \N^2$ such that 
$$
\Vert q\epsilon_0^n \Vert \leq  \frac{\kappa}{\epsilon_0^{2n}   q^{2}}\cdotp
$$
This means that the function 
$$
(q,n)\longmapsto  
\epsilon_0^{2n}   q^{2}\Vert q\epsilon_0^n \Vert 
$$ 
tends to infinity as $\max\{q,n\}$ tends to infinity. Liouville's inequality (Proposition $\ref{Proposition:VarianteLiouville}$) gives only a lower bound for $\epsilon_0^{2n}   q^{2}\Vert q\epsilon_0^n \Vert 
$ with an explicit positive constant. The equality $(\ref{Equation:PC})$ cannot be used  because   $q\epsilon_0^n$ is a pseudo-Pisot number. The conclusion of Theorem $\ref{thm:Effectif}$ is $$
\Vert q\epsilon_0^n \Vert \leq  n^{-\kappa_6 n}q^{-\kappa_6 \log n}
$$
for $n\ge 2$, 
which is weaker than the estimates that we deduced from Theorem  $\ref{thm:principal}$. 

For an explicit example, let $D$ be an integer $>1$ and let $\omega = \root 3 \of {D^3 -1}>1$. 
The fundamental unit $>1$ of the cubic field  $\Q(\omega)$ is (see \cite{S1})  $\epsilon_0 = 1/(D-\omega) =D^2 + D\omega + \omega^2$. 
\end{exa}

\begin{exa} Let $K$ be a number field of degree $d$. Assume that there are two independent real units $\epsilon_2>\epsilon_1>1$ in $K$. Since $\epsilon_1 $, $\epsilon_2 $ are multiplicatively independent, the numbers  $\log\epsilon_1 $,  $\log \epsilon_2 $ are linearly independent over $\Q$, hence $\Z \log\epsilon_1 +\Z \log \epsilon_2 $ is a dense subgroup of $\R$ and therefore  the multiplicative  subgroup of $\R_+^\times$ generated by $\epsilon_1 $, $\epsilon_2 $ is dense. Hence there exists a   sequence of units $\varepsilon_n=\epsilon_1^{a_n} \epsilon_2^{-b_n}$ such that $1/2\le \varepsilon_n\le 2$.
The numbers $a_n$ and $b_n$ are positive integers which tend to infinity. Since
$$
|\log \varepsilon_n| =  |a_n\log \epsilon_1-b_n\log\epsilon_2|\le \log 2,
$$
the   limit of the sequence  $a_n/b_n$  is $(\log\epsilon_2)/(\log \epsilon_1 )$. For instance one can take for $a_n/b_n$ the convergents of the continued fraction expansion of  $(\log\epsilon_2)/(\log \epsilon_1 )$. The sequences 
$$
\left(
\frac{\log \house{\varepsilon_n}}
{a_n\log \epsilon_1}
\right)
_{n\ge 1}
\quad\hbox{and}\quad 
\left(
\frac{\log \house{\varepsilon_n}}
{b_n\log \epsilon_2}
\right)
_{n\ge 1}
$$
converge to the positive limit 
$$
\max_{\sigma:K\rightarrow \C}
\left(
\frac{\log |\sigma \epsilon_1|}
{\log \epsilon_1}
 -
 \frac{\log |\sigma \epsilon_2| }
 {\log \epsilon_2}
\right).
$$
For $n$ sufficiently large, we have $\house{\varepsilon_n} \ge e^e$. 
 Liouville's inequality from Proposition $\ref{Proposition:VarianteLiouville}$ is
$$
(q  \,  \house{\varepsilon_n})^{d-1} \Vert q\varepsilon_n\Vert\ge \kappa_1,
$$
and Theorem $\ref{thm:principal}$ yields 
$$
  (q\,  \house{\varepsilon_n})^{d-1} \Vert q \varepsilon_n\Vert   \longrightarrow +\infty
\quad\hbox{as}\quad \max\{q, n\}   \longrightarrow +\infty,
$$ 
while 
Theorem $\ref{thm:Effectif}$ gives the lower bound 
$$
\Vert q\varepsilon_n\Vert \ge  
\exp
\bigl\{ 
-\kappa_7
\bigl ( 
\log \log\house{\varepsilon_n}
\bigr )
 \log \max\{  q,\; 2\}
 \bigr\},
$$
and  Theorem  $\ref{thm:CZ}$ yields 
$$
 \liminf  \frac{d\log q + \log \Vert q\varepsilon_n\Vert }{\log q +\log\house{\varepsilon_n}} \ge  0 
$$
 as $\max\{q, n \}   \longrightarrow +\infty$ with      $  q\varepsilon_n$ being not a pseudo--Pisot number. 
Hence Theorem $\ref{thm:Effectif}$ is sharper, when $n$ is large,
 than the estimate which one deduces from Theorem  $\ref{thm:principal}$ and than  the estimate which one deduces from Theorem $\ref{thm:CZ}$.

For an explicit example, let $D$ be an integer $>1$ and let $\omega = \root 4 \of {D^4 -1}>1$. 
A pair of independent units of the biquadratic number field $\Q(\omega)$  is (see \cite{S2}) 
$$
\epsilon_1=D^2+\omega^2=\frac{1}{D^2-\omega^2}
\quad\hbox{and}\quad
\epsilon_2 =\frac{1}{D-\omega} =D^3 + D^2\omega + D\omega^2+\omega^3.
$$

\end{exa}

\subsection*{Acknowledgements}
The authors are thankful to Francesco Amoroso, K\'{a}lm\'{a}n Gy\H{o}ry  and Yann Bugeaud for their remarks on a preliminary version of this paper.


 \vfill\vfill
 \bigskip
 
\noindent
{\sc Claude Levesque \hfill \hfill \hfill \hfill \hfill \hfill Michel Waldschmidt }

\end{document}